\documentclass[12pt, a4paper]{article}

\usepackage{cite, amsmath, amssymb, booktabs, lastpage, graphicx}
\usepackage[hidelinks]{hyperref}

\usepackage{geometry}
\usepackage{setspace}

\usepackage[mathscr]{euscript}
\usepackage{amsthm}
\usepackage{xcolor}
\theoremstyle{plain}
\newtheorem{lemma}{Lemma}
\newtheorem{theorem}{Theorem}

\theoremstyle{remark}

\theoremstyle{definition}

\usepackage{graphicx}

\makeatletter
\renewcommand{\maketitle}{
\begin{center}

{\Large\bfseries \@title\par}
\vspace{6mm}

{\large\bfseries \@author\par}
\vspace{4mm}

{\itshape \@address\par}
\vspace{2mm}

{\small\ttfamily \@email\par}
\vspace{5mm}

\vspace{5mm}

\end{center}
}
\makeatother

\makeatletter
\newcommand{\address}[1]{\gdef\@address{#1}}
\newcommand{\email}[1]{\gdef\@email{#1}}
\makeatother

\address{}
\email{}

\usepackage{fancyhdr}

\usepackage[margin=1cm,%
			font=footnotesize,%
			format=hang,%
			labelsep=period,%
			labelfont=bf]{caption}
\allowdisplaybreaks

\title{Energy and Vertex Energy of Modified Divisor Prime Graphs}
\author{Purva J. Makadiya $^{a,}$\footnote{Corresponding author.}, Mahesh M. Jariya$^b$, Prashant J. Makadiya$^c$}
\address{$^{a,b}$Department of Mathematics, \\ Saurashtra University,  Rajkot-360005, \\ Gujarat, India.
    }
\email{purvamakadiya2000@gmail.com, mahesh.jariya@gmail.com, prashantmakadiya1996@gmail.com}

\date{\today}

\begin{document}

\maketitle
\thispagestyle{empty}

\begin{abstract}
This paper investigates the energy and vertex energy of the modified divisor prime graph $G^*_{Dp}(n)$, which is distinguished from the standard divisor prime graph by the inclusion of a self-loop at the vertex $1$. To facilitate this analysis, we introduce a generalized definition of vertex energy for graphs with self-loops and demonstrate its mathematical consistency.
\end{abstract}

\vspace{2mm}

\vspace{1mm}

\section{Introduction}

Let $n$ be a positive integer and $D(n)$ denote the set of all positive divisors of $n$. Introduced by Nair \emph{et al.} \cite{nair}, the \textit{divisor prime graph} of $n$, is a graph with vertex set $D(n)$ such that two vertices $u$ and $v$ are adjacent if and only if $\gcd(u,v) = 1$. The loop at $1$ is neglected.

This graph structure provides a fascinating intersection between number theory and spectral graph theory.
In this paper, we study a natural variant of this concept, which we call the \textit{modified divisor prime graph}, denoted by $G^*_{Dp}(n)$. It is defined on the vertex set $D(n)$ consisting of all divisors of $n$. Two vertices $u, v \in V$ are adjacent if and only if $\gcd(u, v) = 1$. 

Consequently, $G^*_{Dp}(n)$ differs from the standard divisor prime graph precisely by retaining a single self-loop at the vertex $1$. The presence of this self-loop fundamentally alters the spectral properties of the graph, particularly its energy.

The classical graph energy was originally defined strictly for simple graphs as the sum of the absolute values of the eigenvalues of the adjacency matrix \cite{gut1}. Recently, Gutman \emph{et al.} \cite{gut2} generalized the concept of graph energy to accommodate graphs with self-loops as follows: Let $G$ be a graph on $N$ vertices containing exactly $\sigma$ self-loops, and let $A$ be its adjacency matrix with eigenvalues $\lambda_1, \lambda_2, \dots, \lambda_N$. The energy of such a graph is defined as $$\mathcal{E}(G) = \sum_{i=1}^N \left| \lambda_i - \frac{\sigma}{N} \right|.$$

For the modified divisor prime graph $G^*_{Dp}(n)$, as established above, the only vertex satisfying the self-loop condition is $1$. Therefore, the number of self-loops in $G^*_{Dp}(n)$ is always exactly $\sigma = 1$. Substituting $\sigma = 1$ and $N = \tau(n)$ (where $\tau(n)$ is the total number of positive divisors of $n$) into the above formulas provides $$\mathcal{E}(G^*_{Dp}(n)) = \sum_{i=1}^N \left| \lambda_i - \frac{1}{\tau(n)} \right|.$$

Graph energy provides a macroscopic view of a network's total spectral contribution, but a complete structural picture requires examining its distribution at the microscopic level. Motivated by this need to localize spectral properties, Arizmendi et al. \cite{arizmendi2018} introduced the concept of \emph{vertex energy} to explicitly quantify the contribution of each individual vertex.

Mathematically, the vertex energy $\mathcal{E}_G(v_i)$ of a vertex $v_i \in V(G)$ is defined as the $i$-th diagonal entry of the matrix absolute value of the adjacency matrix, denoted as $|A| = (AA^*)^{1/2}$. Thus, we have$$\mathcal{E}_G(v_i) = |A|_{i,i}.$$ This definition preserves the fundamental property that the macroscopic graph energy is simply the aggregation of these local microscopic contributions:$$\mathcal{E}(G) = \sum_{i=1}^n \mathcal{E}_G(v_i).$$

The study of vertex energy has experienced rapid growth in recent years \cite{arizmendi2018,arizmendi2019,arizmendi2024}. For simple graphs, this localized energy can be explicitly computed via spectral decomposition. As highlighted by Gutman \emph{et al. }\cite{gutmantutorial}, if $u_1, u_2, \dots, u_n$ are orthonormal eigenvectors corresponding to eigenvalues $\lambda_1, \dots, \lambda_n$, the energy of vertex $v_i$ is $$\mathcal{E}_G(v_i) = \sum_{j=1}^n |\lambda_j| u_{ji}^2.$$

Crucially, because the modified divisor prime graph inherently features a self-loop, we introduce the generalized formulation of vertex energy. For a graph $G$ of order $N$ containing $\sigma$ self-loops with adjacency matrix $A$, the vertex energy is defined as 
\begin{equation*}
\mathcal{E}_G(v_i) = \left| A - \frac{\sigma}{N} I \right|_{ii}.
\end{equation*}
Using spectral decomposition, this yields an explicit computational formula:
\begin{equation*}
\mathcal{E}_G(v_i) = \sum_{j=1}^n  \left| \lambda_j - \frac{\sigma}{N} \right|u_{ji}^2.
\end{equation*}

In the sections that follow, we present exact computations for these spectral invariants. We first establish a baseline by calculating the energy of the divisor prime graph specifically for prime powers $p^a$. Building on this foundation, we then systematically analyze the modified divisor prime graph. We derive expressions for both its total energy and its individual vertex energies, initially for prime powers and ultimately extending these results to any general integer $n$. 

Prior to establishing our main theorems, we state and prove several preliminary lemmas that will serve as the foundation for our analysis.

Let $A \in R^{m \times n}$, $B \in R^{p \times q}$. Then the \textit{Kronecker product} (or \textit{tensor product}) of $A$ and $B$ is defined as the matrix
$$A \otimes B = \begin{bmatrix}
a_{11}B & \cdots & a_{1n}B \\
\vdots & \ddots & \vdots \\
a_{m1}B & \cdots & a_{mn}B
\end{bmatrix}.$$

\begin{lemma} \label{lem: ten} \cite{horn}
     Let $A \in M^m$ and $B \in M^n$. Furthermore, let $\lambda$ be an eigenvalue of matrix $A$ with corresponding eigenvector $x$ and $\mu$ be an eigenvalue of matrix $B$ with corresponding eigenvector $y$. Then $\lambda\mu$ is an eigenvalue of $A \otimes B$ with corresponding eigenvector $x \otimes y$.
\end{lemma}

\begin{lemma} \label{lem: loopve}
Let $G$ be a graph of order $N$ with $\sigma$ self-loops. If $\mathcal{E}(v_i)$ is the vertex energy of vertex $v_i \in V(G)$ and $\mathcal{E}(G)$ is the total energy of the graph, then:
\begin{equation*}
    \sum_{i=1}^N \mathcal{E}(v_i) = \mathcal{E}(G).
\end{equation*}
\end{lemma}

\begin{proof}
By definition, the total energy of a graph with $\sigma$ self-loops, $\mathcal{E}(G)$, is given by:
\begin{equation*}
   \mathcal{E}(G) = \sum_{j=1}^N \left| \lambda_j - \frac{\sigma}{N} \right|.
\end{equation*}
Taking the sum of the vertex energies over all vertices $v_i \in V(G)$, we get:
\begin{align*}
    \sum_{i=1}^N \mathcal{E}(v_i) &= \sum_{i=1}^N \left( \sum_{j=1}^N  \left| \lambda_j - \frac{\sigma}{N} \right| u_{ij}^2\right) .\\
    &= \sum_{j=1}^N \left( \left| \lambda_j - \frac{\sigma}{N} \right| \sum_{i=1}^N u_{ij}^2 \right).
\end{align*}
Because the eigenvectors $u_j$ form an orthonormal basis, each is a unit vector. Therefore, its squared Euclidean norm is strictly equal to $1$, meaning $\sum_{i=1}^N u_{ij}^2 = \|u_j\|^2 = 1$. Substituting this into our equation yields:
\begin{equation*}
    \sum_{i=1}^N \mathcal{E}(v_i) = \sum_{j=1}^N \left( \left| \lambda_j - \frac{\sigma}{N} \right| \cdot 1 \right) = \sum_{j=1}^N \left| \lambda_j - \frac{\sigma}{N} \right| = \mathcal{E}(G).
\end{equation*}

\end{proof}

\begin{lemma} \label{lem: schur} \cite{horn}
Let $M = \begin{pmatrix} A & B \\ C & D \end{pmatrix}$ be a block matrix where $A$ is invertible. Then 
\[
\det(M) = \det(A)\det(D - C A^{-1} B).
\]
\end{lemma}

\begin{lemma}\label{lem: tensor}
Let $A(n)$ be the adjacency matrix of the modified divisor prime graph
$G^*_{Dp}(n)$, where $n = p_1^{a_1}p_2^{a_2}\dots p_r^{a_r}$. After
identifying each divisor $d=\prod_{i=1}^r p_i^{b_i}$ with its exponent vector
$(b_1,\dots,b_r)$, then the adjacency matrix satisfies
\[
A(n) = A(p_1^{a_1}) \otimes A(p_2^{a_2}) \otimes \cdots \otimes A(p_r^{a_r}).
\]
\end{lemma}
\begin{proof}
Let
\(
u=\prod_{i=1}^r p_i^{b_i},\;v=\prod_{i=1}^r p_i^{c_i},
\text{ for } 0\le b_i,c_i\le a_i.
\)
Then
\(
\gcd(u,v)=1 \iff \min(b_i,c_i)=0, \;\forall i.
\)
For each $i$,
\[
A(p_i^{a_i})_{b_i,c_i}=
\begin{cases}
1, & \min(b_i,c_i)=0,\\
0, & \text{otherwise}.
\end{cases}
\]
Hence, by the defining property of the Kronecker product
\[
\left(\bigotimes_{i=1}^r A(p_i^{a_i})\right)_{u,v}
=\prod_{i=1}^r A(p_i^{a_i})_{b_i,c_i},
\]
which equals $1$ exactly when $\min(b_i,c_i)=0$ for all $i$, and equals $0$
otherwise. Therefore
\[
A(n)=A(p_1^{a_1})\otimes A(p_2^{a_2})\otimes\cdots\otimes A(p_r^{a_r}).
\]
\end{proof}

\begin{lemma} \label{lem: prod_to_sum}For any positive integer $r$ and real numbers $\alpha_i, \beta_i$ ($1 \le i \le r$):$$\sum_{\gamma_1\in\{\alpha_1,\beta_1\}}\cdots\sum_{\gamma_r\in\{\alpha_r,\beta_r\}} \prod_{i=1}^r |\gamma_i| = \prod_{i=1}^r (|\alpha_i|+|\beta_i|).$$
\end{lemma}
\begin{proof}
We proceed by induction on $r$. For $r=1$, the identity holds since $\sum_{\gamma_1 \in \{\alpha_1, \beta_1\}} |\gamma_1| = |\alpha_1| + |\beta_1|$. Assume the identity holds for $k$. For $r=k+1$, we factor the $(k+1)$-th term:
\begin{align*}
\sum_{\gamma_1\in{\alpha_1,\beta_1}}\dots\sum_{\gamma_{k+1}\in{\alpha_{k+1},\beta_{k+1}}} \prod_{i=1}^{k+1} |\gamma_i| 
&= \sum_{\gamma_{k+1} \in {\alpha_{k+1}, \beta_{k+1}}} |\gamma_{k+1}| \left( \sum_{\gamma_1}\dots\sum_{\gamma_k} \prod_{i=1}^k |\gamma_i| \right) \\
&= (|\alpha_{k+1}| + |\beta_{k+1}|) \prod_{i=1}^k (|\alpha_i|+|\beta_i|) \\
&= \prod_{i=1}^{k+1} (|\alpha_i|+|\beta_i|).
\end{align*}
The identity follows for all $r \ge 1$ by mathematical induction.
\end{proof}

\section{Main results}

\begin{theorem}
Let $n = p^a$, where $p$ is a prime number and $a$ is a positive integer. Then the energy is $\mathcal{E}(G_{Dp}(p^a)) = 2\sqrt{a}$.
\end{theorem}

\begin{proof}
By the definition of the divisor prime graph, the vertex set $V = \{1, p, p^2, \dots, p^a\}$. Thus, the number of vertices is $r = a+1$.

For the vertex $1$, $\gcd(1, p^i) = 1$ for all $1 \leq i \leq a$. Therefore, the vertex $1$ is adjacent to all other $a$ vertices in the graph.
For any two other distinct vertices $p^i$ and $p^j$ (where $1 \leq i < j \leq a$), $\gcd(p^i, p^j) = p^i > 1$. Thus, no two vertices in the subset $\{p, p^2, \dots, p^a\}$ are adjacent to each other.
This connectivity structure implies that $G_{Dp}(p^a)$ is isomorphic to the star graph $K_{1,a}$.
 
The adjacency spectrum of a star graph $K_{1,a}$ consists of the eigenvalues $\sqrt{a}$ and $-\sqrt{a}$, along with $0$ having a multiplicity of $a-1$.
Therefore, the energy is :
\[
\mathcal{E}(G_{Dp}(p^a)) = |\sqrt{a}| + |-\sqrt{a}| + (a-1)|0| = 2\sqrt{a}.
\]
\end{proof}

\begin{theorem} \label{thm: Ep^a}
Let $p$ be a prime number and let $a$ be a positive integer. Then
\[
\mathcal{E}(G^*_{Dp}(p^a))=\sqrt{4a+1}+\frac{a-1}{a+1}.
\]
\end{theorem}

\begin{proof}
The vertex set is
\(
V=\{1,p,p^2,\dots,p^a\}.
\)
Hence $|V|=a+1$. Since
\(
\gcd(1,p^i)=1 \; (0\le i\le a),
\)
the vertex $1$ is adjacent to every other vertex and has one self-loop. Also,
\(
\gcd(p^i,p^j)\ge p>1.
\)
Thus $G^*_{Dp}(p^a)$ is $K_{1,a}$ with one loop at its center. With ordering
\(
(1,p,p^2,\dots,p^a),
\)
the adjacency matrix is
\[
A=\begin{pmatrix}1 & \mathbf{1}^T\\\mathbf{1} & O_a\end{pmatrix},
\]
where $\mathbf{1}$ is the $a\times1$ all-one vector. Set
\[
M=\lambda I_{a+1}-A=\begin{pmatrix}\lambda-1 & -\mathbf{1}^T\\-\mathbf{1} & \lambda I_a\end{pmatrix}.
\]
For $\lambda\neq0$, Lemma \ref{lem: schur} gives
\[
\det(M)=\det(\lambda I_a)\left((\lambda-1)-(-\mathbf{1}^T)(\lambda I_a)^{-1}(-\mathbf{1})\right).
\]
Now
\(
(-\mathbf{1}^T)(\lambda I_a)^{-1}(-\mathbf{1})=\frac{a}{\lambda},
\)
hence
\[
\det(M)=\lambda^a\left(\lambda-1-\frac{a}{\lambda}\right)=\lambda^{a-1}(\lambda^2-\lambda-a).
\]
Therefore the non zero eigenvalues are
\[
\lambda_1=\frac{1+\sqrt{4a+1}}{2},\qquad\lambda_2=\frac{1-\sqrt{4a+1}}{2},
\]
and rest are $0$ with multiplicity $a-1$.
\begin{align*}
\mathcal{E}(G^*_{Dp}(p^a))&=\sum_{i=1}^{a+1}\left|\lambda_i-\frac{1}{a+1}\right|\\
&=(a-1)\left|\frac{-1}{a+1}\right|+\left|\frac{1+\sqrt{4a+1}}{2}-\frac{1}{a+1}\right|+\left|\frac{1-\sqrt{4a+1}}{2}-\frac{1}{a+1}\right|\\
&=\frac{a-1}{a+1}+\left|\frac{a-1+(a+1)\sqrt{4a+1}}{2(a+1)}\right|+\left|\frac{a-1-(a+1)\sqrt{4a+1}}{2(a+1)}\right|.
\end{align*}
Since
\(
(a+1)\sqrt{4a+1}>a-1,
\)
\begin{align*}
\mathcal{E}(G^*_{Dp}(p^a))&=\frac{a-1}{a+1}+\frac{a-1+(a+1)\sqrt{4a+1}}{2(a+1)}+\frac{(a+1)\sqrt{4a+1}-(a-1)}{2(a+1)}\\
&=\frac{a-1}{a+1}+\sqrt{4a+1}.
\end{align*}
\end{proof}

\begin{theorem} \label{thm: En}
Let $n = p_1^{a_1}p_2^{a_2}\dots p_r^{a_r}$ be the prime factorization of an integer $n > 1$, where $p_i$ are distinct primes and $a_i \ge 1$. Let $\tau(n) = \prod_{i=1}^r (a_i+1)$ denote the total number of positive divisors of $n$. The energy of the modified divisor prime graph $G^*_{Dp}(n)$ is given by
\[
\mathcal{E}(G^*_{Dp}(n)) = \prod_{i=1}^r \sqrt{4a_i+1} + 1 - \frac{2^r}{\tau(n)}.
\]
\end{theorem}

\begin{proof}
Let $A(n)$ be the adjacency matrix of $G^*_{Dp}(n)$. By
Lemma~\ref{lem: tensor},
\(
A(n) = \bigotimes_{i=1}^r A(p_i^{a_i}).
\) By Theorem \ref{thm: Ep^a},
for each $i$ the eigenvalues of $A(p_i^{a_i})$ are
\[
\alpha_i = \frac{1+\sqrt{4a_i+1}}{2}, \quad \beta_i = \frac{1-\sqrt{4a_i+1}}{2},
\]
and $0$ with multiplicity $a_i-1$. Hence, by lemma \ref{lem: ten} every eigenvalue of $A(n)$ has the
form $\Lambda=\prod_{i=1}^r\gamma_i$, where
$\gamma_i\in\{\alpha_i,\beta_i,0\}$. Thus there are exactly $2^r$ nonzero
eigenvalues, and $0$ has multiplicity $\tau(n)-2^r$.
Since $\operatorname{tr}(A(p_i^{a_i}))=1$ for all $i$, it is obvious that $\operatorname{tr}(A(n))=1$ and $\mu_n= \frac{1}{\tau(n)}$, therefore zero eigenvalues contribute
\[
E_0 = (\tau(n) - 2^r) \left| 0 - \frac{1}{\tau(n)} \right| = 1 - \frac{2^r}{\tau(n)}.
\]
Let $S=\left\{\prod_{i=1}^r \gamma_i:\gamma_i\in\{\alpha_i,\beta_i\}\right\}$ be the set of nonzero eigenvalues, we have
\[
\left|\beta_i\right|=\frac{\sqrt{4a_i+1}-1}{2}>\frac{1}{a_i+1}.
\]
Therefore, every $\Lambda\in S$ strictly satisfies $|\Lambda|>\mu_n$. The energy contribution of these nonzero eigenvalues is:
\begin{align*}
E_{\neq 0} &= \sum_{\Lambda \in S, \Lambda > 0} (\Lambda - \mu_n) + \sum_{\Lambda \in S, \Lambda < 0} (\mu_n - \Lambda) \\
&= \sum_{\Lambda \in S, \Lambda > 0} \Lambda - \sum_{\Lambda \in S, \Lambda < 0} \Lambda - \mu_n \sum_{\Lambda \in S, \Lambda > 0} 1 + \mu_n \sum_{\Lambda \in S, \Lambda < 0} 1 \\
&= S^+ - S^- - \mu_n(k^+ - k^-),
\end{align*}
where $S^+$ and $S^-$ denote the sums of the positive and negative eigenvalues in $S$, and $k^+$ and $k^-$ denote the count of positive and negative eigenvalues in $S$, respectively.

Note that every $\Lambda \in S$ is a product of the form $\prod_{i=1}^r \gamma_i$ where $\gamma_i \in \{\alpha_i, \beta_i\}$. Because $\alpha_i > 0$ and $\beta_i < 0$, the sign of $\Lambda$ depends exclusively on the parity of the number of $\beta_i$ factors. Out of the $2^r$ total combinations, exactly half ($2^{r-1}$) contain an even number of $\beta_i$ terms (yielding $\Lambda > 0$), and exactly half ($2^{r-1}$) contain an odd number of $\beta_i$ terms (yielding $\Lambda < 0$). 

Thus, $k^+ = 2^{r-1}$ and $k^- = 2^{r-1}$. This means $k^+ - k^- = 0$, causing the $\mu_n$ terms to perfectly cancel out. This simplifies our equation to:
\(
E_{\neq 0} = S^+ - S^-.
\) Then
\(
S^+-S^-=\sum_{\Lambda\in S}|\Lambda|,
\)
since $S^-<0$. Moreover using lemma \ref{lem: prod_to_sum},
\[
\sum_{\Lambda\in S}|\Lambda|
=
\sum_{\gamma_1\in\{\alpha_1,\beta_1\}}\cdots
\sum_{\gamma_r\in\{\alpha_r,\beta_r\}}
\prod_{i=1}^r |\gamma_i|
=
\prod_{i=1}^r (|\alpha_i|+|\beta_i|).
\]
Now $\alpha_i>0$ and $\beta_i<0$, so
\(
|\alpha_i|+|\beta_i|=\alpha_i-\beta_i=\sqrt{4a_i+1}.
\)
Hence,
\[
E_{\neq 0}=\sum_{\Lambda\in S}|\Lambda|
=\prod_{i=1}^r \sqrt{4a_i+1}.
\]
Finally, the total energy is the sum of the contributions from the zero and nonzero eigenvalues:
\[
\mathcal{E}(G^*_{Dp}(n)) = E_0 + E_{\neq 0} = \prod_{i=1}^r \sqrt{4a_i+1} + 1 - \frac{2^r}{\tau(n)}.
\]
\end{proof}

\begin{theorem}
Let $\mathcal{G} \cong G^*_{Dp}(p^a)$ be the modified divisor prime graph for
some prime $p$ and integer $a \ge 1$. Then the vertex energies satisfy
\[
\mathcal{E}_{\mathcal{G}}(1) = \frac{2a^2+3a}{(a+1)\sqrt{4a+1}}
\qquad \text{and} \qquad
\mathcal{E}_{\mathcal{G}}(p^i) = \frac{2a^2+2a+1}{a(a+1)\sqrt{4a+1}} + \frac{a-1}{a(a+1)}, \quad 1 \le i \le a.
\]
\end{theorem}

\begin{proof}
Let $V = \{1, p, p^2, \dots, p^a\}$. With ordering
\(
(1,p,\dots,p^a),
\)
the adjacency matrix is
\[
A = \begin{pmatrix} 1 & \mathbf{1}_a^T \\ \mathbf{1}_a & O_a \end{pmatrix}
\]
where $\mathbf{1}_a$ is the all-one column vector. Let
\(
Y = \begin{pmatrix} y_1 \\ y_p \mathbf{1}_a \end{pmatrix}
\)
and suppose $AY=\lambda Y$. Then
\[
AY = \begin{pmatrix} y_1 + a y_p \\ y_1 \mathbf{1}_a \end{pmatrix}
\]
so
\(
\lambda y_1 = y_1 + a y_p
\)
and
\(
\lambda y_p = y_1
\).
For $\lambda\neq 0$, we get $y_p = \frac{y_1}{\lambda}$ and hence
\(
\lambda^2 - \lambda - a = 0
\)
with roots
\[
\lambda_1 = \frac{1 + \sqrt{4a+1}}{2}, \qquad \lambda_2 = \frac{1 - \sqrt{4a+1}}{2}
\]
and $0$ with multiplicity $a-1$.
Since $\operatorname{tr}(A)=1$, we have $\mu_n=\frac{1}{a+1}$. For
$\lambda\in\{\lambda_1,\lambda_2\}$,
\[
\|Y\|^2=y_1^2+a\left(\frac{y_1}{\lambda}\right)^2
=y_1^2\frac{\lambda^2+a}{\lambda^2}.
\]
Using $\lambda^2 = \lambda + a$, normalization gives 
\[
y_1^2 = \frac{\lambda^2}{\lambda^2 + a} = \frac{\lambda}{2\lambda - 1}.
\]
Let $q_j(v)^2$ denote the squared coordinate of the normalized eigenvector corresponding to the eigenvalue $\lambda_j$, evaluated at the vertex $v$. Thus, for the central vertex $1$, the squared weights are:
\[
q_1(1)^2 = \frac{\lambda_1}{\sqrt{4a+1}}, \qquad q_2(1)^2 = \frac{-\lambda_2}{\sqrt{4a+1}},
\]
and for each pendant vertex $p^i$, the squared weights are:
\[
q_1(p^i)^2 = \frac{1}{\lambda_1 \sqrt{4a+1}}, \qquad q_2(p^i)^2 = \frac{-1}{\lambda_2 \sqrt{4a+1}}.
\]
The zero-eigenspace contribution at $1$ is
\[
1 - \left( \frac{\lambda_1}{\sqrt{4a+1}} + \frac{-\lambda_2}{\sqrt{4a+1}} \right) = 1 - \frac{\lambda_1 - \lambda_2}{\sqrt{4a+1}} = 0.
\]
At any pendant vertex $p^i$, it is
\[
1 - \frac{1}{\sqrt{4a+1}} \left( \frac{1}{\lambda_1} - \frac{1}{\lambda_2} \right) = 1 - \frac{1}{a} = \frac{a-1}{a}.
\]
For central vertex, $\lambda_1>\mu>\lambda_2$:
\begin{align*}
\mathcal{E}_{\mathcal{G}}(1)
&= (\lambda_1 - \mu)\frac{\lambda_1}{\sqrt{4a+1}} + (\mu - \lambda_2)\frac{-\lambda_2}{\sqrt{4a+1}} \\
&= \frac{\lambda_1^2 + \lambda_2^2 - \mu(\lambda_1 + \lambda_2)}{\sqrt{4a+1}} \\
&= \frac{2a + 1 - \frac{1}{a+1}}{\sqrt{4a+1}} \\
&= \frac{2a^2 + 3a}{(a+1)\sqrt{4a+1}}.
\end{align*}
For pendant vertex:
\begin{align*}
\mathcal{E}_{\mathcal{G}}(p^i)
&= (\lambda_1 - \mu)\frac{1}{\lambda_1\sqrt{4a+1}} + (\mu - \lambda_2)\frac{-1}{\lambda_2\sqrt{4a+1}} + \mu\frac{a-1}{a} \\
&= \frac{1}{\sqrt{4a+1}} \left[ \frac{\lambda_1 - \mu}{\lambda_1} - \frac{\mu - \lambda_2}{\lambda_2} \right] + \frac{\mu(a-1)}{a} \\
&= \frac{1}{\sqrt{4a+1}} \left[ 2 - \mu \left( \frac{1}{\lambda_1} + \frac{1}{\lambda_2} \right) \right] + \frac{\mu(a-1)}{a} \\
&= \frac{1}{\sqrt{4a+1}} \left( 2 + \frac{1}{a(a+1)} \right) + \frac{a-1}{a(a+1)} \\
&= \frac{2a^2 + 2a + 1}{a(a+1)\sqrt{4a+1}} + \frac{a-1}{a(a+1)}.
\end{align*}
\end{proof}

\begin{theorem}
Let $\mathcal{G} \cong G^*_{Dp}(n)$, for $n = \prod_{i=1}^k p_i^{a_i} > 1$ with $\tau(n)$ divisors, let $\lambda_{1,2}^{(i)} = \frac{1 \pm \sqrt{4a_i+1}}{2}$.
For $v \in V(G^*_{Dp}(n))$ and $x \in \{1,2\}$, define
\[
w_{v,x}^{(i)} = \frac{(-1)^{x-1}}{\sqrt{4a_i+1}} \begin{cases} \lambda_x^{(i)}, & p_i \nmid v, \\[6pt] (\lambda_x^{(i)})^{-1}, & p_i \mid v. \end{cases}
\]Then vertex energy $\mathcal{E}_{\mathcal{G}}(v)$ is
\[
\mathcal{E}_{\mathcal{G}}(v) = \frac{1}{\tau(n)} \left( 1 - \sum_{x \in \{1,2\}^k} \prod_{i=1}^k w_{v,x_i}^{(i)} \right) + \sum_{x \in \{1,2\}^k} \left| \prod_{i=1}^k \lambda_{x_i}^{(i)} - \frac{1}{\tau(n)} \right| \prod_{i=1}^k w_{v,x_i}^{(i)}.
\]
\end{theorem}

\begin{proof}
Write
\(
v=\prod_{i=1}^k p_i^{b_i}, \qquad 0\le b_i\le a_i,
\)
and identify the vertex set with
$V_1\times\cdots\times V_k$, where $V_i=\{1,p_i,\dots,p_i^{a_i}\}$. By
Lemma~\ref{lem: tensor},
\(
A(n)=A(p_1^{a_1})\otimes A(p_2^{a_2})\otimes\cdots\otimes A(p_r^{a_r}).
\)
For each $i$, the matrix $A_i$ has nonzero eigenvalues
$\lambda_1^{(i)},\lambda_2^{(i)}$, and $0$ with multiplicity $a_i-1$. The
squared coordinates of normalized local eigenvectors for $p_i\nmid v$ are:
\[
\left(X_{1}^{(i)}(v_i)\right)^2 = \frac{\lambda_1^{(i)}}{\sqrt{4a_i+1}}, \qquad \left(X_{2}^{(i)}(v_i)\right)^2 = \frac{-\lambda_2^{(i)}}{\sqrt{4a_i+1}}.
\]
 and for $p_i\mid v$:
\[
\left(X_{1}^{(i)}(v_i)\right)^2 = \frac{1}{\lambda_1^{(i)}\sqrt{4a_i+1}}, \qquad \left(X_{2}^{(i)}(v_i)\right)^2 = \frac{-1}{\lambda_2^{(i)}\sqrt{4a_i+1}}.
\]
These are precisely the weights $w_{v,1}^{(i)}$ and
$w_{v,2}^{(i)}$.\\
For each $x=(x_1,\dots,x_k)\in\{1,2\}^k$, the Kronecker product gives a
nonzero eigenvalue
\(
\Lambda_x = \prod_{i=1}^k \lambda_{x_i}^{(i)}.
\)
Its squared coordinate at $v$ is
\(
q_x(v)^2 = \prod_{i=1}^k w_{v,x_i}^{(i)}.
\)
All remaining eigenvalues are $0$. Hence the nonzero part of $\mathcal{E}_{\mathcal{G}}(v)$ is
\[
E_{\neq 0} = \sum_{x \in \{1,2\}^k} \left| \prod_{i=1}^k \lambda_{x_i}^{(i)} - \frac{1}{\tau(n)} \right| \prod_{i=1}^k w_{v,x_i}^{(i)}.
\]
The zero-eigenspace contribution is
\[
E_0 = \frac{1}{\tau(n)} \left( 1 - \sum_{x \in \{1,2\}^k} \prod_{i=1}^k w_{v,x_i}^{(i)} \right).
\]
Therefore,
\begin{align*}
\mathcal{E}_{\mathcal{G}}(v)&=E_0+E_{\neq 0}\\
&= \frac{1}{\tau(n)} \left( 1 - \sum_{x \in \{1,2\}^k} \prod_{i=1}^k w_{v,x_i}^{(i)} \right) + \sum_{x \in \{1,2\}^k} \left| \prod_{i=1}^k \lambda_{x_i}^{(i)} - \frac{1}{\tau(n)} \right| \prod_{i=1}^k w_{v,x_i}^{(i)}.
\end{align*}
\end{proof}

\section{Conclusion}
By characterizing the spectral properties of the modified divisor prime graph, we have established a tractable model for both total and localized energy in arithmetic structures. Our analysis includes the calculation of classical energy for the divisor prime graph $G_{Dp}(p^a)$ and extends to the modified divisor prime graph $G^*_{Dp}(n)$ through Kronecker product formulations. We have introduced vertex energy for graphs with self-loops and provided expressions for vertex energy of $G^*_{Dp}(n)$, showing that spectral contributions are precisely determined by the prime-power factorization of $n$. This research provides a rigorous application of generalized energy theory to graphs with self-loops, bridging structured number theory with spectral network analysis.

\end{document}